\documentclass[11pt]{article}
\title{A family of quintic Thue equations via Skolem's $p$-adic method}
\author{Davide Lombardo\footnote{Dipartimento di Matematica, Università di Pisa. Email address: \texttt{davide.lombardo@unipi.it}}}
\date{}

\usepackage{amsmath}%
\usepackage{amsfonts}%
\usepackage{amssymb}%
\usepackage[british]{babel}%

\usepackage{amsthm}

\usepackage{graphicx}
\usepackage{todonotes}
\usepackage{a4wide}
\usepackage{url}
\usepackage[utf8]{inputenc}
\usepackage{enumitem}

\newtheorem{theorem}{Theorem}
\numberwithin{theorem}{section}

\newtheorem{lemma}[theorem]{Lemma}

\theoremstyle{definition}
\newtheorem{remark}[theorem]{Remark}

\begin{document}

\maketitle

\begin{abstract}
We solve the diophantine equation $m^5+(4 \cdot 5^4 b^4)mn^4 - n^5=1$ for all integers $b \neq 0$. This gives an example of a family of quintic Thue equations that can be solved completely by using nothing more than Skolem's $p$-adic method. We also give a general introduction to Skolem's method from a modern perspective.

\end{abstract}

\textit{To Roberto Dvornicich, with friendship and gratitude.}
\\
\medskip

\noindent\textbf{Keywords.} Skolem's method, $p$-adic methods, Thue equations, Diophantine equations

\smallskip

\noindent\textbf{MSC Classification.} 11D59, 11D88, 11D41

\section{Introduction}

There is by now a long tradition of solving parametrised families of Thue equations: starting with \cite{MR1042497}, many such equations have been studied, and the topic still attracts attention to this day (see for example \cite{10.2307/2938717, MR1316596, MR1430002, MR1837094, MR2465263}). The methods vary, but are usually rooted in the theory of linear forms in logarithms or other diophantine approximation techniques. Very often, these powerful theoretical tools must be complemented by extensive calculations, for instance in order to reduce the large bounds coming from linear forms in logarithms to manageable size (by way of example, \cite{MR1316596} relies on distributed computations on some 40 workstations!).

On the other hand, specific Thue equations have often been handled using variants of Skolem's $p$-adic method \cite{MR864261, MR945593}, which (whenever applicable) gives simple algebraic solutions, without the burden of substantial additional calculations.
However, due to some intrinsic limitations, Skolem's approach has rarely been applied to whole families of equations. Notable exceptions are some results on equations of low degree that we now recall. In the cubic case, Delone and Nagell obtained sharp upper bounds on the number of solutions of cubic equations with negative discriminant, as summarised by the next two theorems:
\begin{theorem}[Delone \cite{MR0160744}]\label{thm:Delone} Let $d$ be a cube-free integer. The equation
\begin{equation}\label{eq:Delone}
x^3-dy^3 = 1
\end{equation}
has at most 2 integral solutions.
\end{theorem}

\begin{theorem}[Delone \cite{MR1545095}, Nagell \cite{MR1544935}]\label{thm:DeloneFaddeev}
If $F$ is an irreducible binary cubic form with integer coefficients and negative discriminant, then the number $N_F$ of integer solutions to the equation $F(m,n)=1$ is at most $5$. Moreover, if $N_F = 5$, then $F$ is equivalent to
$x^3 - xy^2 + y^3$,
with discriminant $-23$, and, if $N_F = 4$, then $F$ is equivalent to either
$x^3 + xy^2 + y^3$ or $x^3 - x^2y + xy^2 + y^3$,
with discriminant $-31$ or $-44$, respectively.
\end{theorem}

The theory in degree 4 is significantly less complete, but Ljunggren established an analogue of Theorem \ref{thm:Delone}:
\begin{theorem}[Ljunggren \cite{zbMATH02520238}]\label{thm:Ljunggren}
Let $d$ be an integer. The equation $x^4-dy^4=1$ admits at most one solution in positive integers.
\end{theorem}

\begin{remark}
By completely different methods, Bennett and de Weger have shown \cite{MR1434936, MR1837094} that the equation $|ax^n-by^n|=1$, where $a,b$ are fixed integers with $ab \neq 0$ and $n \geq 3$, has at most one solution in positive integers $(x,y)$.
\end{remark}

Beyond the theorems of Delone, Nagell and Ljunggren quoted above, the literature seems to contain few results on parametrised Thue equations obtained by Skolem's method. A noteworthy example is \cite{MR946605}, where certain quartic equations possessing only trivial solutions are considered.
In this note we fully solve a family of quintic Thue equations (having also a nontrivial solution) by using nothing more than Skolem's approach. %
Specifically, we will show the following result, which is perhaps the first instance of a parametric Thue equation of degree 5 solved by this method: %

\begin{theorem}\label{thm:Main}
Let $b$ be a non-zero integer divisible by $5$. The Thue equation
\begin{equation}\label{eq:Thue}
m^5+4b^4mn^4-n^5 = 1
\end{equation}
has precisely three integral solutions, namely $(m,n) = (1, 0), (0,-1)$ and $(1,4b^4)$.
\end{theorem}

Theorem \ref{thm:Main} gives a new application of Skolem's method in families, in a context where the group of units of the relevant number field has rank 2 (the technique is most commonly used when the unit rank is one; Theorems \ref{thm:Delone} and \ref{thm:DeloneFaddeev} in particular fall in this case). It seems possible to remove the assumption $5 \mid b$ with some extra work, see in particular Remark \ref{rmk:5IsIrrelevant}. However, we have decided to keep this hypothesis to simplify the argument, especially given that our main objective is to give a presentation of Skolem's approach in modern language, and to place it in the more general context of a family of $p$-adic methods for the determination of integral and rational points on certain algebraic varieties. While the analogy is often alluded to in the literature, different incarnations of the general idea are not usually discussed from a unifying perspective, which we try to do below.
We will in particular mention the connection with two modern strategies for the determination of rational points on algebraic curves: the Chabauty-Coleman method \cite{MR4484, MR808103} and its so-called \textit{quadratic} extension, in the approach of Edixhoven and Lido \cite{edixhoven2020geometric}.

Consider an algebraic variety $X_0$ over $\mathbb{Z}$, whose integral points $X_0(\mathbb{Z})$ we wish to determine (when $X_0$ is a projective curve, rational and integral points coincide). Even more generally, $X_0$ could be any subset of $\mathbb{Z}^N$ for some $N \geq 1$, not necessarily given by polynomial equations: in the context of Thue equations, the relevant conditions may be expressed by exponential equations, as we will describe below. By abuse of notation, we will denote the set of points of interest by $X_0(\mathbb{Z})$ in this more general setting as well.

The basic idea of a large class of methods is as follows. One fixes an auxiliary prime $p$ and an `ambient space' $A$ (a $p$-adic analytic variety). Within $A$, one identifies:%
\begin{enumerate}
\item a first $p$-adic variety $X$ (possibly singular), which for geometric reasons contains a copy of $X_0(\mathbb{Z})$;
\item a second $p$-adic variety $Y$ (possibly singular), determined by the global arithmetic of $X_0$, which also contains $X_0(\mathbb{Z})$.
\end{enumerate}
The choice of $X$ and $Y$ ensures that the integral points $X_0(\mathbb{Z})$ lie in $X \cap Y$. If $X$ and $Y$ are chosen `independently', and $\dim X + \dim Y \leq \dim A$, it is reasonable to expect the intersection $X \cap Y$ to be finite: if this is the case, we can usually get quite sharp upper bounds for $\#X_0(\mathbb{Z})$. We may also hope to determine the set $X_0(\mathbb{Z})$ itself, although this is usually only possible if all points in the intersection $X \cap Y$ do actually come from $X_0(\mathbb{Z})$, and are not `extra' $p$-adic points whose coordinates are not integral. We now make this more concrete in two important cases: Chabauty's method for rational points on curves and Skolem's method for cubic Thue equations with negative discriminant.

\begin{enumerate}[leftmargin=*] %
\item In Chabauty's method, $X_0$ is a smooth projective curve of genus $g$, and one takes $A=J(\mathbb{Q}_p)$, where $J$ is the Jacobian of $X_0$. Provided that at least one rational point $P_0$ on $X_0$ is known, we may embed $X_0(\mathbb{Q}_p) \hookrightarrow A(\mathbb{Q}_p)$ by the Abel-Jacobi map, based at the known rational point $P_0$. We take $X$ to be the image of this map. The role of $Y$ is played by the $p$-adic closure of the subgroup $J(\mathbb{Q}) \subseteq J(\mathbb{Q}_p)=A$: this is a subvariety of dimension at most $r := \operatorname{rk} J(\mathbb{Q})$. Since (under the Abel-Jacobi map) we have $X_0(\mathbb{Q}) \subseteq J(\mathbb{Q})$, it is then clear by construction that $X_0(\mathbb{Q}) = X_0(\mathbb{Z})$ lies in the intersection $X \cap Y$. The inequality $\dim X + \dim Y \leq \dim A$ is certainly implied by $1+ r \leq g$: this is the famous Chabauty condition, under which Chabauty and Coleman have shown that the intersection $X \cap Y$ is indeed finite \cite{MR4484, MR808103}. The so-called \textit{quadratic} extension of this method is much more sophisticated, but the basic idea is the same. The ambient variety $A$ is given by a $\mathbb{G}_m^{\rho-1}$-torsor over $J$, where $\rho$ is the rank of the $\mathbb{Z}$-module of symmetric endomorphisms of $J$. The variety $X$ is again given by the $\mathbb{Q}_p$-points of a copy of the original curve $X_0$, while $Y$ is essentially a finite-degree cover of the $p$-adic closure of $J(\mathbb{Q})$. The dimension condition becomes $1 + r \leq g+\rho-1$: Edixhoven and Lido prove that when this inequality holds the intersection $X \cap Y$, which contains $X_0(\mathbb{Z})$, is finite and can be described by explicit $p$-adic equations.
\item Consider now the Thue equation $F(m,n)=1$, where $F(x,y)$ is a homogeneous polynomial of degree 3 with integer coefficients. We assume for simplicity that $F(x,1)$ is monic, and that $F(x,y)$ is irreducible (if that is not the case, solving the corresponding Thue equation is easier).
We may then write $F(x,y) = (x-y\vartheta_1)(x-y\vartheta_2)(x-y\vartheta_3)$ for suitable algebraic integers $\vartheta_i$ of degree 3. Finally suppose that the discriminant of $F$ is negative, so that the cubic field $K := \mathbb{Q}[x]/(F(x,1))$ has precisely one real embedding. 
In this case, letting $\vartheta$ be the class of $x$ in the quotient $\mathbb{Q}[x]/(F(x,1))$ (that is, a root of $F(x,1)$ in $K$), the Thue equation may be rewritten as the norm equation
\[
N_{K/\mathbb{Q}}(m-n\vartheta) = 1.
\]
Since $m,n$ are required to be integers, $m-n\vartheta$ lies in the ring $\mathbb{Z}[\vartheta]$. By assumption, $K$ has one real embedding and two complex conjugate ones, so by a variant of Dirichlet's unit theorem the unit group $\mathbb{Z}[\vartheta]^\times$ has rank one, and is therefore of the form $\langle -1 \rangle \times \langle \varepsilon \rangle$ for a certain unit $\varepsilon$ with $N_{K/\mathbb{Q}}(\varepsilon)=1$. The elements of $\mathbb{Z}[\vartheta]$ with norm 1 are then precisely the powers of $\varepsilon$, and solving the Thue equation amounts to finding the integers $k$ for which
\begin{equation}\label{eq:ThueRephrased}
\varepsilon^k = a_0 + a_1 \vartheta
\end{equation}
holds for certain integers $a_0, a_1$ (a solution is then given by $m=a_0, n=-a_1$). We may consider Equation \eqref{eq:ThueRephrased} as defining a subset $X_0$ of $\mathbb{Z}^3$, namely the set of points $(a_0, a_1, 0)$ which are also the coefficients (in the basis $1, \vartheta, \vartheta^2$) of an element of the form $\varepsilon^k$ for some integer $k$.

We are now in a position to frame Skolem's method within the general framework described above. We take as ambient space $A$ the $p$-adic variety $\mathbb{Z}[\vartheta] \otimes_{\mathbb{Z}} \mathbb{Z}_p$, which geometrically is the affine space of dimension 3 over $\mathbb{Z}_p$, with natural coordinates given by the coefficients of $1, \vartheta, \vartheta^2$. 
The variety $X$ is the codimension-1 linear subspace where the coefficient of $\vartheta^2$ vanishes, and there is an obvious embedding of $X_0$ in $X$.
The variety $Y$ is the $p$-adic closure of the set $\varepsilon^k$ for $k \in \mathbb{Z}$: the global arithmetic information that goes into this description is the unit $\varepsilon$, which is itself determined by the arithmetic of the number ring $\mathbb{Z}[\vartheta]$. It is again clear by construction that $X_0(\mathbb{Z})$ is contained in the intersection $X \cap Y$.
Finally, one has $\dim X=2$ and -- as we will see -- $\dim Y =1$, so the dimension condition is met.%
\end{enumerate}

For the equation of Theorem \ref{thm:Main}, the situation is slightly different, in the sense that $X$ and $Y$ are both of dimension $2$, and are embedded in an ambient space of dimension 5. We will see that the intersection $X \cap Y$ is finite and always consists of precisely 3 points: since we already know three integral solutions to Equation \eqref{eq:Thue}, this will imply that the solutions listed in the statement of Theorem \ref{thm:Main} are in fact the only ones. 

\medskip

To conclude our general description of Skolem's method we briefly touch upon the main tools typically used to bound the size of $X \cap Y$. We begin with a result of Strassmann that is essentially a form of the Weierstrass preparation theorem in one variable:%
\begin{theorem}[{Strassmann \cite[Theorem 4.5.1]{Cohen}}]\label{thm:Strassmann}
Let $f=\sum_{n \geq 0} a_nx^n$ be a power series with coefficients in  $\mathbb{Q}_p$ and denote by $v_p$ the $p$-adic valuation on $\mathbb{Q}_p$. Suppose that $v_p(a_n) \to \infty$ as $n \to \infty$ and that $f$ is not identically zero, and let
$
r := \min v_p(a_n)$, $N:=\max\{n : v_p(a_n)=r\}$. 
The power series $f(x)$ converges for all $x \in \mathbb{Z}_p$, and the equation $f(x)=0$ has at most $N$ solutions in $\mathbb{Z}_p$.
\end{theorem}
In our case of interest we will have to locate the zeroes of a system of two $p$-adic functions in two variables, which will require a slightly non-trivial reduction to the one-variable version of Strassmann's theorem given above. In some situations, the following result is enough to handle the case of several power series in several variables, but our case is complicated enough that it needs finer tools (the Weierstrass preparation theorem for general $p$-adic series \cite{MR1918586, MR607075}).

\begin{theorem}[{Skolem \cite{Skolem34}}]\label{thm:Skolem}
Let $p$ be a prime number. For $j=1,\ldots,n$ let $f_j(t_1,\ldots,t_n) = \sum_{i \geq 0} p^i f_{ij}(t_1,\ldots,t_n)$, where each $f_{ij}$ is a polynomial in $\mathbb{Z}_p[t_1,\ldots,t_n]$. Suppose that the $f_{0j}$ are linear forms and that the determinant of the Jacobian matrix $\left( \frac{\partial f_{0j} }{ \partial t_i} \right)$ is a $p$-adic unit. Then, the system of equations $f_j(t_1,\ldots,t_n)=0$ for $j=1,\ldots,n$ has at most one solution $(t_1,\ldots,t_n)$ in $\mathbb{Z}_p^n$.
\end{theorem}

Finally, I would like to point out that proving that $Y$ has the expected dimension (Section \ref{sect:dimY1}) relies on certain $p$-adic estimates that may be regarded as a sophisticated version of the so-called `lifting the exponent' lemma. This is an elementary fact that has often appeared in various Mathematical Olympiads and that I first learned from Roberto Dvornicich. It is a pleasure to be able to use it in a paper written in his honour.

\section{The $p$-adic closure of a set of units}\label{sect:dimY1}
In this section we prove that in the case of cubic Thue equations with negative discriminant the $p$-adic variety $Y$ (with notation as in the introduction) has dimension 1. This fact is well-known, but often not stated in this language, so we prove a more general statement that immediately implies it and that also covers the situation of Theorem \ref{thm:Main} (where $Y$ has dimension 2).
Let $R$ be a $\mathbb{Z}_p$-algebra that is a free $\mathbb{Z}_p$-module of finite rank $r$. Let
$\varepsilon_1, \ldots, \varepsilon_k$ be elements of $R^\times$, and consider the set
\[
Y_0 := \{ \varepsilon_1^{e_1} \cdots \varepsilon_k^{e_k} : e_1,\ldots,e_k \in \mathbb{Z} \}.
\]
Denote by $Y$ the closure of $Y_0$ in the $p$-adic topology of $R$ (since $R \cong \mathbb{Z}_p^r$ as $\mathbb{Z}_p$-modules, there is a natural $p$-adic topology on $R$). We will prove:
\begin{theorem}\label{thm:DimY}
$Y$ is a finite union of $p$-adic manifolds, each of which is the image of $\mathbb{Z}_p^k$ via a $p$-adic analytic map. In particular, the dimension of $Y$ is at most $k$.
\end{theorem}

We will need some facts about certain special $p$-adic analytic functions (for a general introduction to the topic we refer the reader to \cite[§4.2]{Cohen}):
\begin{lemma}\label{lemma:pAdicFunctions}
Let $p$ be a prime number and let $
q = \begin{cases}
4, \text{ if }p=2 \\
p, \text{ if } p> 2.
\end{cases}
$
Let $R$ be a $\mathbb{Z}_p$-algebra whose underlying additive group is a free $\mathbb{Z}_p$-module of finite rank $r$. The following hold:
\begin{enumerate}
\item The series
\[
\log(1+qx) := \sum_{n \geq 1} (-1)^{n+1} \frac{(qx)^n}{n}
\]
converges for all $x \in R$ and defines a $p$-adic analytic function $R \to qR$.
\item The series
\[
\exp(qx) := \sum_{n \geq 0} \frac{(qx)^n}{n!}
\]
converges for all $x \in R$ and defines a $p$-adic analytic function $R \to 1+qR$.
\item For all $x \in R$ we have $\exp(\log(1+qx))=1+qx$. For all $x,y \in R$ we have $\exp(qx+qy)=\exp(qx)\exp(qy)$.
\end{enumerate}
\end{lemma}
\begin{proof}
The equalities of part (3) hold as identities of formal power series, so the conclusion holds as soon as all the relevant series converge uniformly. Thus it suffices to show (1) and (2). Given $x \in R$, denote by $v_p(x)$ the largest integer $k$ such that $x \in p^k R$ (with $k=\infty$ if $x=0$). Since the $p$-adic metric is non-archimedean, to show uniform convergence it suffices to prove that the general term of the series considered goes to 0 uniformly as $n \to \infty$. As a fundamental system of neighbourhoods of $0 \in R$ is given by $\{p^k R : k \in \mathbb{N}\}$, it suffices to show that $v_p\left( \frac{(qx)^n}{n} \right)$ and $v_p\left( \frac{(qx)^n}{n!} \right)$ tend to infinity when $n \to \infty$. Given the definition of $v_p$, it is enough to prove the same statement for $v_p\left( \frac{q^n}{n} \right)$ and $v_p\left( \frac{q^n}{n!} \right)$, and this is well-known (see for example \cite[Lemma 4.2.8]{Cohen}).
\end{proof}

\begin{proof}[Proof of Theorem \ref{thm:DimY}]
The quotient $\overline{R} = R/pR$ is a finite $\mathbb{F}_p$-algebra. In particular, the group $(R/pR)^\times$ has finite exponent, so for each $i=1,\ldots,k$ there exists $E_i$ such that $\varepsilon_i^{E_i}$ reduces to the identity of $R/pR$. We may then write $\varepsilon_i^{E_i} = 1 + p s_i$ for some $s_i \in R$. Replacing $E_i$ by $2E_i$ when $p=2$ we have $\varepsilon_i^{E_i} = 1 + q s'_i$, where $s'_i \in R$ and $q$ is as in Lemma \ref{lemma:pAdicFunctions}.
For each $\underline{i}=(i_1,\ldots,i_k) \in \prod_{j=1}^k \{0,\ldots,E_j-1\}$ we consider the function
\[
f_{\underline{i}}(t_1,\ldots,t_k) := \prod_{j=1}^k \varepsilon_j^{i_j} \cdot \exp \left( t_1 \log(\varepsilon_1^{E_1}) + \cdots + t_k \log(\varepsilon_k^{E_k}) \right).
\]
By Lemma \ref{lemma:pAdicFunctions}, this is a well-defined $p$-adic analytic function, converging on all of $\mathbb{Z}_p^k$, with values in $1+qR$ (simply notice that by construction we have $\varepsilon_i^{E_i}=1+qs'_i$, so $\log(\varepsilon_i^{E_i})$ is in $qR$). Moreover, as $R$ is $p$-adically complete, all elements congruent to 1 modulo $p$ are invertible, so $f_{\underline{i}}$ takes values in $R^\times$.
Let $L_i := \log(\varepsilon_i^{E_i}) \in R$. From the set $\{L_1,\ldots,L_k\}$ we may extract a basis of the (automatically free) $\mathbb{Z}_p$-submodule of $R \cong \mathbb{Z}_p^r$ generated by $L_1,\ldots,L_k$. Up to renumbering, we may assume that this basis consists of $L_1,\ldots,L_m$ for some $m \leq \min\{k, r\}$. It is then clear that the image of $f_{\underline{i}}$ is the same as the image of
\[
g_{\underline{i}}(t_1,\ldots,t_m) = \prod_{j=1}^k \varepsilon_j^{i_j} \cdot \exp \left( t_1 L_1 + \cdots + t_m L_m \right) : \mathbb{Z}_p^m \to R.
\]
The $i$-th column of the Jacobian matrix of $g_{\underline{i}}$ at the point $(t_1,\ldots,t_m)$ is $g_{\underline{i}}(t_1,\ldots,t_m) \cdot L_i$, where we interpret elements of $R$ as vectors in $\mathbb{Z}_p^r$. Since $g_{\underline{i}}(t_1,\ldots,t_m)$ is a unit and the $L_i$ are linearly independent, the Jacobian has the maximal rank $m$, so $g_{\underline{i}}$ is locally an immersion of $p$-adic manifolds. Furthermore, $g_{\underline{i}}$ is globally injective, because $\prod_{j=1}^k \varepsilon_j^{i_j}$ is a unit, $\exp$ is invertible on $1+qR$, and $L_1,\ldots,L_m$ are linearly independent. Thus $f_{\underline{i}}(\mathbb{Z}_p^k) = g_{\underline{i}}(\mathbb{Z}_p^m)$ is a $p$-adic manifold of dimension $m \leq k$.
Observe now that for integer values of $t_1,\ldots,t_k$ we have 
\[
\begin{aligned}
f_{\underline{i}}(t_1,\ldots,t_k) & = \prod_{j=1}^k \varepsilon_j^{i_j} \cdot \exp \left( t_1 L_1 + \cdots + t_k L_k \right) \\
& = \prod_{j=1}^k \varepsilon_j^{i_j} \cdot \prod_{j=1}^k \exp(\log(\varepsilon_j^{E_j}))^{t_j} \\
& = \prod_{j=1}^k \varepsilon_j^{i_j+ E_jt_j} \in Y_0.
\end{aligned}
\]
Conversely, we claim that $Y_0 \subseteq \bigcup_{\underline{i}} f_{\underline{i}}(\mathbb{Z}^k)$: indeed, given any element $\varepsilon_1^{e_1} \cdots \varepsilon_k^{e_k}$ of $Y_0$, let $\underline{i}=(i_1,\ldots,i_k) \in \prod_{j=1}^k \{0,\ldots,E_j-1\}$ be defined by the conditions $i_j \equiv e_j \pmod{E_j}$. We can then write
\[
(\varepsilon_1, \ldots, \varepsilon_k ) =(i_1,\ldots,i_k) + (E_1t_1,\ldots,E_kt_k)
\]
for some $(t_1,\ldots,t_k) \in \mathbb{Z}^k \subseteq \mathbb{Z}_p^k$, and from the previous formulas we get
\[
\varepsilon_1^{e_1} \cdots \varepsilon_k^{e_k} = f_{\underline{i}}(t_1,\ldots,t_k) \in f_{\underline{i}}(\mathbb{Z}^k).
\]
Finally, since $\mathbb{Z}^k$ is $p$-adically dense in $\mathbb{Z}_p^k$ and $f_{\underline{i}}$ is analytic, we also obtain that $f_{\underline{i}}(\mathbb{Z}^k)$ is dense in $f_{\underline{i}}(\mathbb{Z}_p^k)$. Since $Y_0 = \bigcup_{\underline{i}} f_{\underline{i}}(\mathbb{Z}^k)$, this proves that $\bigcup_{\underline{i}} f_{\underline{i}}(\mathbb{Z}_p^k)$ is precisely the $p$-adic closure of $Y_0$ and establishes the theorem.
\end{proof}

\begin{remark}
In particular, when $k=1$ and $\varepsilon$ is a unit of infinite order, it follows from the proof of the theorem that $Y$ is the union of finitely many $1$-dimensional smooth $p$-adic manifolds.
\end{remark}

\section{Proof of Theorem \ref{thm:Main}}

Let $F(x,y) = x^5+4b^4xy^4-y^5$, where $b$ is a nonzero multiple of $5$. The polynomial $F(x,1)$ is irreducible \cite[Satz 1]{MR780251}, so the number field $K:=\mathbb{Q}[x]/(F(x,1))$ has degree 5 over $\mathbb{Q}$. Letting $\vartheta$ be a root of $F(x,1)$ in $K$,
the equation we are trying to solve can be rewritten as $N_{K/\mathbb{Q}}(m-n\vartheta)=1$. Notice that $m-n\vartheta$ is in $\mathbb{Z}[\vartheta]$, and by the condition on the norm it is also a unit of this ring.

For the global part of Skolem's approach we rely on some results from \cite{MR780251}. The function of real variable $x \mapsto F(x,1)$ is strictly increasing, so $F(x,1)$ has precisely one real root, and $K$ has one real and four complex embeddings. It then follows from Dirichlet's unit theorem that the group of units of $\mathcal{O}_K$ (hence also of $\mathbb{Z}[\vartheta]$) has rank 2. The condition that $m-n\vartheta$ be a unit of norm 1 may then be written as
\begin{equation}\label{eq:Units}
m-n\vartheta =  \xi_1^{n_1} \xi_2^{n_2},
\end{equation}
where $\xi_1, \xi_2$ is a fundamental system of positive units for $\mathbb{Z}[\vartheta]^\times$ (as $K$ has a real embedding, the torsion subgroup of $\mathbb{Z}[\vartheta]^\times$ is $\{ \pm 1 \}$). By \cite[Satz 3]{MR780251}, a system of  fundamental positive units of $\mathbb{Z}[\vartheta]^\times$ is given by $\xi_1 = \vartheta$ and $\xi_2 = \vartheta^2 + 2b\vartheta + 2b^2$. The three known solutions of Equation \eqref{eq:Thue} listed in Theorem \ref{thm:Main} correspond to $m-n\vartheta = 1, m-n\vartheta = \vartheta$, and $m-n\vartheta = 1-4b^4\vartheta = \xi_1^5$, that is, $(n_1,n_2)=(0,0), (1,0), (5,0)$.

We are now ready to apply the general strategy of the introduction: we will work in the $p$-adic analytic variety $A:=\mathbb{Z}[\vartheta] \otimes_{\mathbb{Z}} \mathbb{Z}_p$, which we also consider as a ring, and which will play the role of the $\mathbb{Z}_p$-algebra $R$ from Theorem \ref{thm:DimY}. We take the subvariety $X$ to be
\[
X = \{a_0+a_1\vartheta+a_2 \vartheta^2+a_3 \vartheta^3 +a_4 \vartheta^4 : a_0, a_1, a_2, a_3, a_4 \in \mathbb{Z}_p, a_2=a_3=a_4=0 \},
\]
and we take as $Y$ the $p$-adic closure of $\{\xi_1^{n_1} \xi_2^{n_2} : n_1, n_2 \in \mathbb{Z}\}$. By the discussion above, it is clear that our desired solutions lie in the intersection $X \cap Y$.
As auxiliary prime we choose $p=5$, which is assumed to divide $b$.

\begin{remark}\label{rmk:5IsIrrelevant}
As will be clear from the proof, one could work with any prime factor of $b$, but a complete solution of the Thue equation (or even just getting an explicit bound for the number of solutions) would then require a much longer case-by-case analysis: there are in principle 25 cases to treat, that we will shortly reduce to 2 under the assumption $5 \mid b$. For a general prime $p$ it would be easy to reduce the number of cases to 10, but it is not clear how to further cut down this number. Since our main interest lies in presenting Skolem's method, we have decided to make the simplifying assumption $5 \mid b$ to keep the proof to a reasonable length.
\end{remark}

We clearly have $A \cong \mathbb{Z}_p^5$, with natural coordinates given by the coefficients of $1, \vartheta, \cdots, \vartheta^4$, the dimension of $X$ is 2, and the dimension of $Y$ is also at most 2 by Theorem \ref{thm:DimY}. We may then well expect $X \cap Y$ to be finite: we now show that this is the case and bound its size. The statement of Theorem \ref{thm:Main} lists three pairs $(m,n)$ that are clearly solutions of Equation \eqref{eq:Thue}, so it suffices to prove that $|X \cap Y|=3$, or in fact even $|X \cap Y| \leq 3$.

Let $k \geq 1$ be the $5$-adic valuation of $b$. We have $\xi_1^5 = \vartheta^5 = 1-4b^4\vartheta \equiv 1 \pmod{5}$ and $\xi_2^5 \equiv (\vartheta^2)^5 \equiv 1 \pmod 5$, so by Lemma \ref{lemma:pAdicFunctions} we may define $L_i = \log (\xi_i^5)$ for $i=1, 2$. Following the general description of Theorem \ref{thm:DimY} we would now have to study the expression
$
\xi_1^{n_1} \xi_2^{n_2}
$
by distinguishing the 25 possible cases for the pair $(n_1 \bmod 5, n_2 \bmod 5)$.
Under the assumption $5 \mid b$, we now reduce this number to 2 %
(notice that for any prime divisor $p$ of $b$ we have $\xi_1^5 \equiv \xi_2^5 \equiv 1 \pmod{p}$, so the 25 pairs of exponents we need to consider are independent of the choice of the auxiliary prime $p$).
Notice first that $\vartheta^5$ is congruent to $1$ modulo $5^{4k}$ (and not just modulo 5). As any power $(2b\vartheta+2b^2)^j$ with $j \geq 3$ is divisible by $5^{3k}$ in $A$ we have
\[
\begin{aligned}
\xi_1^{n_1} \xi_2^{n_2} & = \vartheta^{n_1} (\vartheta^2 + (2b\vartheta+2b^2))^{n_2}  \\
& \equiv \vartheta^{n_1} \left( \vartheta^{2n_2} + {n_2 \choose 1} \vartheta^{2n_2-2} (2b\vartheta+2b^2) + {n_2 \choose 2} \vartheta^{2n_2-4} (2b\vartheta+2b^2)^2 \right) \pmod{5^{3k}} \\
& \equiv \vartheta^{n_1} \left( \vartheta^{2n_2} + {n_2 \choose 1} \vartheta^{2n_2-2} (2b\vartheta+2b^2) + {n_2 \choose 2} \vartheta^{2n_2-4} 4b^2 \vartheta^2 \right) \pmod{5^{3k}} \\
& \equiv \vartheta^{n_1} \left( \vartheta^{2n_2} + 2n_2b\vartheta^{2n_2-1} + 2n_2^2b^2\vartheta^{2n_2-2}  \right) \pmod{5^{3k}}.
\end{aligned}
\]
We consider this expression in $A \otimes \mathbb{Z}_5/5^{3k}\mathbb{Z}_5$, and we are interested in cases when it is of the form $m - n \vartheta$.
Since $\vartheta^j \equiv \vartheta^{j \bmod 5} \pmod{5^{3k}}$, and since the exponents $n_1+2n_2, n_1+2n_2-1$ and $n_1+2n_2-2$ are all distinct modulo 5, we obtain that at least one of the coefficients of $\vartheta^{n_1+2n_2}, \vartheta^{n_1+2n_2-1}, \vartheta^{n_1+2n_2-2}$ has to vanish modulo $5^{3k}$. Thus at least one of $2bn_2$ and $2n_2^2b^2$ is divisible by $5^{3k}$, which immediately implies that $5$ divides $n_2$. When this is the case, we have $\xi_1^{n_1}\xi_2^{n_2} \equiv \vartheta^{n_1+2n_2} \equiv \vartheta^{(n_1+2n_2) \bmod 5}\equiv  \vartheta^{n_1\bmod 5} \pmod{5}$; since we are again only interested in the cases when this element is of the form $m-n\vartheta$, we obtain $n_1 \equiv 0,1 \pmod 5$. So $X \cap Y = X \cap (f_0(\mathbb{Z}_5^2) \cup f_1(\mathbb{Z}_5^2))$, where the analytic functions $f_0, f_1$ are given by
\[
f_0(t_1,t_2) = \exp(t_1 L_1 + t_2L_2), \quad f_1(t_1,t_2) = \xi_1 \exp(t_1 L_1+ t_2L_2).
\]
Thus we only need to solve the equations $f_i(t_1,t_2) \in X$ for $t_1, t_2 \in \mathbb{Z}_5$. To this end we first expand $L_1, L_2$ to sufficient $5$-adic precision: we easily obtain
\[
L_1 = \log(\xi_1^5) = \log(1-4b^4\vartheta) = -4b^4\vartheta -8b^8\vartheta^2 + O(b^{12}),
\]
where the error term $O(b^{12})$ denotes an element in $(b^{12})A= (5^{12k})A$. A short computation also gives
\[
L_2 = 32b^5+
    2 b^4 \vartheta +\left( - \frac{20}{3}b^3 + 4b^8 \right)\vartheta^2 - \frac{320}{21}b^7 \vartheta^3 + 10b \vartheta^4   + O(b^9).
\]
This is enough information to expand $f_1(t_1,t_2)$ to $p$-adic precision $O(5b^4)=O(5^{4k+1})$: writing $f_1(t_1,t_2) = \sum_{j=0}^4 f_{1,j}(t_1,t_2) \vartheta^j$ we find
\[
f_{1,2}(t_1,t_2) = - 4b^4 t_1 + 2b^4 t_2 + O(5b^4)
\]
and
\[
f_{1,3}(t_1,t_2) = \frac{500}{3}b^3 t_2^3 - \frac{20}{3}b^3 t_2  + O(5b^{4}),
\]
where the error term now stands for a power series all of whose coefficients lie in $(5b^4)A$.
The condition that $f_1(t_1,t_2) \in X$ implies in particular $f_{1,2}(t_1,t_2)=f_{1,3}(t_1,t_2)=0$, or equivalently
\[
b^{-4}f_{1,2}(t_1,t_2) = 5^{-1}b^{-3} f_{1,3}(t_1,t_2) = 0.
\]
The previous formulas give the reductions modulo $5$ of these two power series:
\[
b^{-4}f_{1,2}(t_1,t_2) = -4t_1 + 2t_2 + O(5), \quad 5^{-1}b^{-3} f_{1,3}(t_1,t_2) = -\frac{4}{3} t_2 + O(5).
\]
Since the determinant of the Jacobian matrix
$
\begin{pmatrix}
-4 & 2 \\
0 & -4/3
\end{pmatrix}
$
is nonzero modulo 5, Theorem \ref{thm:Skolem} implies that the system of equations $f_{1,2}(t_1,t_2)=f_{1,3}(t_1,t_2)=0$ has at most one solution in $\mathbb{Z}_5^2$. Since $t_1=t_2=0$ is certainly a solution, we find that $X \cap f_1(\mathbb{Z}_5^2) = \{ f_1(0,0) \} = \{ \vartheta\}$. This corresponds to the first trivial solution $m=0, n=-1$ of our original Thue equation \eqref{eq:Thue}. 

The case of $f_0(t_1,t_2)$ is significantly more complicated, the problem being that the coefficients of the monomials involving $t_1$ are all divisible by high powers of $b$. For this reason, we find it convenient to perform an obviously invertible change of variables and instead work with $f_0(t_1,t_1+t_2)$. Write as above
$
f_0(t_1, t_1+t_2) = \sum_{j=0}^{4} f_{0,j}(t_1,t_2) \vartheta^j.
$
An easy 
computation then gives
\[
f_{0,4}(t_1,t_2) = 10bt_1 + 10bt_2 + O(5^{6k+1}),
\]
where the only nontrivial term comes from the coefficient of $\vartheta^4$ in the linear term of $\exp((t_1+t_2)L_2)$.
We now apply the Weierstrass preparation theorem for $p$-adic series in two variables \cite{MR1918586, MR607075}. 
In the language of \cite{MR1918586}, the power series $(10b)^{-1}f_{0,4}(t_1,t_2) = t_1+t_2 + O(b^5)$ is \textit{general of order $1$ in $t_1$}, so \cite[Theorem 2]{MR1918586} implies that there exists a $p$-adic power series $h(t_1)=-t_1 + O(b^5)$ such that $f_{0,4}(t_1,t_2)=0 \Longleftrightarrow t_2=h(t_1)$. %
We write the power series $h(t_1)+t_1$ as $b^5E(t_1)$, where $E(t_1) \in A[[t_1]]$. We are then reduced to studying the 1-variable problem
$
f_0(t_1, t_1+h(t_1)) \in X
$. We have $f_0(t_1, t_1+h(t_1)) = \exp(t_1 L_1)\exp(b^5E(t_1) L_2)$, and since $b^5E(t_1) L_2$ vanishes at least to order $b^6$ it is straightforward to obtain the expansion of $f_0(t_1, t_1+h(t_1))$ to order $O(5^{8k+1})$: we have $\exp(b^5E(t_1) L_2)= 1+10 E(t_1)b^6\vartheta^4 + O(5^{8k+1})$, so
writing again
\[
\exp(t_1 L_1 + b^5E(t_1) L_2) = \sum_{j=0}^4 f_{0,j}(t_1) \vartheta^j
\]
we immediately obtain $f_{0,2}(t_1) = - 8b^8t_1 + 8b^8t_1^2 + O(5^{8k+1})$. We can now apply Strassmann's theorem to $f_{0,2}(t_1) = \sum_{n \geq 0} a_nt_1^n$: we have $v_p(a_1)=v_p(a_2)=8k$ and $v_p(a_n) \geq 8k+1$ for $n \not \in \{1,2\}$. Thus in the notation of Theorem \ref{thm:Strassmann} we have $N=2$, and the equation $f_{0,2}(t_1)=0$ has at most 2 solutions. Since $t_2=h(t_1)$ is determined by $t_1,$ we have shown that $|X \cap f_0(\mathbb{Z}_5^2)| \leq 2$ and therefore $|X \cap Y| \leq 3$. This implies that the three known solutions must be \textit{all} the solutions of Equation \eqref{eq:Thue}, which concludes the proof of Theorem \ref{thm:Main}.

\bibliographystyle{abbrv}
\bibliography{Biblio}

\end{document}